\documentclass[a4paper,11pt]{article}
\usepackage{amssymb,amsthm,amsmath,latexsym}

\newtheorem{thm}{Theorem}[section]
\newtheorem{pro}[thm]{Proposition}
\newtheorem{lem}[thm]{Lemma}
\newtheorem{cor}[thm]{Corollary}

\newtheorem{exa}[thm]{Example}

\DeclareMathOperator*{\Psl}{PSL}

\DeclareMathOperator*{\Sz}{Sz}

\date{}

\title{A nilpotency criterion for some verbal subgroups}

\author{\textsc{Carmine Monetta} \;and \textsc{Antonio Tortora}\footnote{The authors are members of {\em National Group for Algebraic and Geometric Structures, and their Applications} (GNSAGA--INdAM).}\\
\small{Dipartimento di Matematica, Universit\`a di Salerno}\\
\small{Via Giovanni Paolo II, 132 - 84084 - Fisciano (SA), Italy}\\
\small{E-mail: cmonetta@unisa.it, antortora@unisa.it}}

\begin{document}
\maketitle

\begin{abstract} 
The word $w=[x_{i_1},x_{i_2},\dots,x_{i_k}]$ is a simple commutator word if $k\geq 2, i_1\neq i_2$ and $i_j\in \{1,\dots,m\}$, for some $m>1$. For a finite group $G$, we prove that if $i_{1} \neq i_j$ for every $j\neq 1$, then the verbal subgroup corresponding to $w$ is nilpotent if and only if  $|ab|=|a||b|$ for any $w$-values $a,b\in G$ of coprime orders. We also extend the result to a residually finite group $G$, provided that the set of all $w$-values in $G$ is finite.\\

\end{abstract}

\section{Introduction}

Let $F$ be the free group on free generators $x_1,\dots,x_m$, for some $m>1$. A group-word is any nontrivial element of $F$, that is, a product of finitely many $x_i$'s and their inverses. The elements of the commutator subgroup of $F$ are called commutator words. We say that the commutator word 
$$[x_{i_1},x_{i_2},\dots,x_{i_k}]=[\dots[[x_{i_1},x_{i_2}],x_{i_3}],\dots,x_{i_k}]$$ 
is a {\em simple commutator word} if $k\geq 2, i_1\neq i_2$ and $i_j\in \{1,\dots,m\}$ for every $j\in \{1,\dots,k\}$. Examples of simple commutator words are the lower central words and the $n$-Engel word 
$$[x,\,_{n} y]=[x, \underbrace{y,\dots,y}_{n}].$$ 

Let $w=w(x_1,\dots,x_k)$ be a group-word in the variables $x_1,\dots, x_k$. For any group $G$  and arbitrary $g_1,\dots, g_k\in G$, the elements of the form $w(g_1,\dots, g_k)$ are called the $w$-values in $G$. We denote by $G_w$ the set of all $w$-values in $G$. The verbal subgroup of $G$ corresponding to $w$ is the (normal) subgroup $w(G)$ of $G$ generated by $G_w$. If $w(G)=1$, then $w$ is said to be a law in $G$. 

Recently the following question has been considered (\cite{BS}, see also \cite{baubau}): 

\noindent {\em Let $w$ be a commutator word and let $G$ be a finite group with the property that $|ab|=|a||b|$ for any $a,b\in G_w$ of coprime orders. Is the verbal subgroup $w(G)$ nilpotent?}
Here $|x|$ stands for the order of the element $x\in G$. 

As remarked in \cite{BS}, by a result of Kassabov and Nikolov \cite{kani} this is not true in general (see Example \ref{v10}). Two easier counterexamples are given in Section \ref{Examples}. On the other hand, the answer to the above question is positive when $w$ is a lower central word \cite{BS,bms}. Motivated by this, in the present paper we prove the following nilpotency criterion for $w(G)$, where $w$ is a simple commutator word without any repetition of the first variable.

\begin{thm}\label{main}	
Let $w=[x_{i_1}, \ldots, x_{i_k}]$ be a simple commutator word with $i_{1} \neq i_j$ for every $j \in \{2, \ldots, k\}$, and let $G$ be a finite group. Then $w(G)$ is nilpotent if and only if $|ab|=|a||b|$ for any $a,b\in G_w$ of coprime orders. 
\end{thm}

We also extend Theorem \ref{main} to a residually finite group $G$, provided that the set of all $w$-values in $G$ is finite. 

\begin{cor}\label{res} \label{residual}
Let $w = [x_{i_1},\ldots, x_{i_k}]$ be a simple commutator word with $i_{1} \neq i_j$ for every $j \in \{2, \ldots, k\}$, and let $G$ be a residually finite group in which $G_w$ is finite. Then $w(G)$ is finite, and it is nilpotent if and only if $|ab| = |a||b|$ for any $a, b \in G_w$ of coprime orders.
\end{cor}

Recall that a group is residually finite if the intersection of its subgroups of finite index is trivial. Notice also that in Corollary \ref{res} the finiteness of $G_w$ depends on the conciseness of the word $w$ in the class of residually finite groups (see Section \ref{procor}). This is related to a question of P.\,Hall on words assuming only finitely many values in a group (see \cite[Part 1, p. 119]{rob2} for an account).

\section{Proof of Theorem \ref{main}}

The aim of this section is to prove the ``if'' part of Theorem \ref{main}, being the ``only if'' part clear. We will split the proof in two cases, depending on whether the finite group $G$ is soluble or not. In particular, we will show that only the soluble case can occur.

\subsection{The soluble case}

If $G$ is a finite soluble group, the Fitting height of $G$ is the least integer $h$ such that $F_h(G)=G$, where $F_0(G)=1$ and $F_{i}(G)/F_{i-1}(G)=F(G/F_{i-1}(G))$ is the Fitting subgroup of $G/F_{i-1}(G)$ for every $i\geq1$. A finite soluble group with Fitting height at most $2$ is said to be metanilpotent. 

The following lemma is well-known (see \cite[Lemma 3]{bms} for a proof). 

\begin{lem}\label{meta} Let $G$ be a finite metanilpotent group with Fitting subgroup $F(G)$. For $p$ a prime, denote by $O_{p'}(G)$ the maximal normal subgroup of $G$ of order coprime to $p$. If $x\in G$ is a $p$-element such that $[O_{p'}(F(G)), x]=1$, then $x\in F(G)$. 
\end{lem}

A subgroup $H$ of a finite group $G$ is called a tower of height $h$ if $H$ can be written as a product $H=P_1\cdots P_h$, where
\begin{enumerate}
	\item[(1)] $P_i$ is a $p_i$-group ($p_i$ a prime) for $i=1,\dots,h$;
	\item[(2)] $P_i$ normalizes $P_j$ for $i<j$;
	\item[(3)] $[P_i,P_{i-1}]=P_i$ for $i=2,\dots,h$.
\end{enumerate}
It follows from (3) that $p_i\neq p_{i+1}$ for $i=1,\dots,h-1$. 

The next lemma is taken from \cite[Lemma 1.9]{turull}.

\begin{lem}\label{Turull}
A finite soluble group $G$ has Fitting height at least $h$ if and only if $G$ has a tower of height $h$.
\end{lem}

Given two nonempty subsets $X$ and $Y$ of a group $G$, let 
\[[X, \, _{n} Y]=[[X, \, _{n-1} Y], Y]\]
where $n\geq 2$ and $[X,Y]$ is the commutator subgroup of $X$ and $Y$. We write $[X,\, _{n} y]$ when $Y=\{y\}$. Then, assuming that $X$ is normalized by $y$, it easy to see that 
\[[X, \, _{n} y]=[X, \, _{n} \langle y\rangle]\]
for every $n\geq 1$.

As a straightforward corollary of \cite[Theorem 5.3.6]{go}, we have:

\begin{lem}\label{Gor}
For $p$ a prime, let $P$ be a $p$-subgroup of a finite group $G$. Suppose that $P$ is normalized by an element $x\in G$ of $p'$-order. Then $$[P,x]=[P,x,x].$$
\end{lem}

\begin{lem}\label{bbb}
Let $w=[x_{i_1}, \ldots, x_{i_k}]$ be a simple commutator word with $i_{1} \neq i_j$ for every $j \in \{2, \ldots, k\}$, and let $G$ be a finite group in which  $|ab|=|a||b|$ for any $w$-values $a,b\in G$ of coprime orders. For $p$ a prime, let $P$ be a $p$-subgroup of $G$ normalized by a $w$-value $x \in G$ of $p'$-order.  Then $[P, x] = 1$.
\end{lem}

\begin{proof}
By Lemma \ref{Gor}, \[[P, x^{-1}]=[P,\, _{k-1} x^{-1}];\] thus the result will follow once it is shown that $N=[P,\, _{k-1} x^{-1}]=1$.

Let $[g, \, _{k-1} x^{-1}] \in N$, for some $g \in P$. Of course the orders of the $w$-values $x$ and $[g, \, _{k-1} x^{-1}]$ are coprime. Then, by hypothesis,
\[|[g, \, _{k-1} x^{-1}] x| = |[g, \, _{k-1} x^{-1}]| |x|. \]
However
\[ [g, \, _{k-1} x^{-1}] x= [g, \, _{k-2} x^{-1}]^{-1} x [g, \, _{k-2} x^{-1}] \] 
is a conjugate of $x$. So $|[g, \, _{k-1} x^{-1}] x|= |x|$ and consequently $[g, \, _{k-1} x^{-1}]=~1$.	
\end{proof}

\begin{lem}\label{directprod}
Let $w=[x_{i_1}, \ldots, x_{i_k}]$ be a simple commutator word and let $G = A \times B$ be an arbitrary group. Then $w(G) = w(A) \times w(B)$.
\end{lem}
\begin{proof} 
By induction on $n$, we get
\[[a_{i_1}b_{i_1},\ldots , a_{i_k}b_{i_k}] =[a_{i_1},\ldots , a_{i_k}][b_{i_1},\ldots , b_{i_k}] \]
for every $a_{i_1},\ldots, a_{i_k} \in A$, and every $b_{i_1},\ldots, b_{i_k} \in B$.
\end{proof}

We are now able to prove the announced result, for soluble groups.

\begin{pro}\label{solu} 
Let $w=[x_{i_1}, \ldots, x_{i_k}]$ be a simple commutator word with $i_{1} \neq i_j$ for every $j \in \{2, \ldots, k\}$, and let $G$ be a finite soluble group in which  $|ab|=|a||b|$ for any  $a,b\in G_w$ of coprime orders. Then $w(G)$ is nilpotent.
\end{pro}

\begin{proof}
Let $h$ be the Fitting height of $G$. Firstly, we show that $h \leq 2$. Suppose by way of contradiction that $h \geq 3$. Then, by Lemma \ref{Turull}, there exists a tower \[P_1P_2P_3 \cdots P_h\] of height $h$ in $G$. Since $P_2 = [P_2, P_1]$ and $P_3 = [P_3, P_2]$, we have \[ P_3 = [P_3, [P_2, P_1]].\] Furthermore, by Lemma \ref{Gor}
\[[P_2, x] = [P_2, \,_{k-1} x]]\]
for every $x \in P_1$. Hence $[P_2, x]$ is generated by $w$-values of $p_2$-orders. Applying Lemma \ref{bbb}, we deduce that $P_3$ commutes with $[P_2, x]$. Thus \[[P_3, [P_2, P_1]] = ~1,\] which is impossible.

Now let $h = 2$, the case $h = 1$ being obvious. Denote by $F$ the Fitting subgroup of $G$. If $w(G) \leq F$, we are done. Suppose that $w(G)$ is not contained in $F$. Since $G/F$ is nilpotent, by Lemma \ref{directprod} there exists a Sylow $p$-subgroup $P$ of $G$ such that $w(P/F)$ = $w(P)F/F$ is nontrivial. Let $x \in w(P)$ be a $w$-value which does not belong to $F$. Then $[O_{p'}(F), x] = 1$, by Lemma \ref{bbb}, from which it follows that $x \in F$, by Lemma \ref{meta}: a contradiction.
\end{proof}

\subsection{The general case}

The following is a well-known consequence of the Baer-Suzuki Theorem (see, for instance, \cite[Theorem 2.13]{Isaacs}). 

\begin{lem}\label{rembearsuz}
Let $G$ be a finite non-abelian simple group. If $x$ is an element of $G$ of order $2$, then there exists $g \in G$ such that $[x,g]$ has odd prime order. 
\end{lem}

In the sequel  we will require a property of finite simple groups whose proper subgroups are soluble. These groups have been classified by Thompson in \cite{thompson}, and they are known as finite minimal simple groups.

\begin{pro}\label{msg}
Let $G$ be a finite minimal simple group. Then $G$ contains a subgroup $H= A \rtimes T$ where $A$ is an elementary abelian $2$-group and $T$ is a subgroup of odd order such that $C_A(T)=1$.
Further, $A=[A,T]$. 
\end{pro}

\begin{proof}
According to Thompson's classification \cite[Corollary 1]{thompson}, the group $G$ is isomorphic to one of the following groups:
\begin{enumerate}
\item[(1)] $\Psl(2, 2^p)$, where $p$ is any prime;
\item[(2)] $\Psl(2, 3^p)$, where $p$ is any odd prime;
\item[(3)] $\Psl(2, p)$, where $p > 3$ is any prime such that $p^2 + 1 \equiv 0 \mod 5$;
\item[(4)] $\Psl(3, 3)$;
\item[(5)] $\Sz(2^p)$, where $p$ is any odd prime.
\end{enumerate}
Since the groups in (2), (3) and (4) have a subgroup isomorphic to the alternating group of degree $4$ (see, for instance, \cite[Theorem 6.26]{suzuki} and \cite[Theorem 7.1(2)]{bloom}), we may consider the other two cases. 

If $G$ is isomorphic to $\Psl(2,2^p)$, then Theorem 6.25 of \cite{suzuki} shows that $G$ contains a Frobenius group $H= A \rtimes T$ where $A$ is an elementary abelian $2$-group of order $q$, and $T$ is a cyclic group of order $q-1$.

If $G$ is isomorphic to the Suzuki group $\Sz(2^p)$, then $G$ contains a Frobenius group $F=Q \rtimes T$ where $Q$ is a Sylow $2$-subgroup of $G$ of order $2^{2p}$ and $T$ is a cyclic subgroup of order $2^p-1$ (see \cite[Theorem 9]{suzuki2}). Thus, taking $A$ to be a minimal normal subgroup of $F$ contained in $Q$, the subgroup $H=A \rtimes T$ is as required.

Finally, notice that in both cases we have $A=[A,T]$ by  \cite[Theorem 5.2.3]{go}.
\end{proof}

\begin{lem}\label{generators} 
Let $w=[x_{i_1}, \ldots, x_{i_k}]$ be a simple commutator word with $i_{1} \neq i_j$ for every $j \in \{2, \ldots, k\}$, and let $G$ be a finite group such that $G=G'$. If $q\in\pi(G)$, then $G$ is generated by $w$-values of $p$-power order for primes $p\neq q$. 
\end{lem}

\begin{proof} For each prime $p\in\pi(G)\setminus\{q\}$, denote by $N_p$ the subgroup of $G$ generated by all $w$-values of $p$-power order. Let us show that each Sylow $p$-subgroup of $G$ is contained in $N_p$. Suppose this is false and choose $p$ such that a Sylow $p$-subgroup of $G$ is not contained in $N_p$. Of course, $N_p$ is a normal subgroup of $G$. We may pass to the quotient $G/N_p$ and assume that $N_p=1$. Since $G=G'$, it is clear that $G$ does not possess a normal $p$-complement. Thus the Frobenius Theorem (see \cite[Theorem 7.4.5]{go}) implies  that $G$ has a $p$-subgroup $H$ and a $p'$-element $a\in N_G(H)$ such that $[H,a]\neq1$. By Lemma \ref{Gor}, we have
\[1\neq[H,a]=[H,\,_{k-1}a]\leq N_p,\] a contradiction. Hence $N_p$ contains the Sylow $p$-subgroups of $G$. Let $T$ be the product of all subgroups $N_p$, with $p\neq q$. Then $G/T$ is a $q$-group and, since $G=G'$, we conclude that $G=T$. It follows that $G$ can be generated by $w$-values of $p$-power order for $p \neq q$. 
\end{proof}

In order to complete the proof of Theorem \ref{main}, we recall that if a simple commutator word is a law in a finite group $G$, then $G$ is nilpotent \cite{hup}. 

\begin{pro}\label{non-solu} 
Let $w=[x_{i_1}, \ldots, x_{i_k}]$ be a simple commutator word with $i_{1} \neq i_j$ for every $j \in \{2, \ldots, k\}$, and let $G$ be a finite group in which  $|ab|=|a||b|$ for any $a,b\in G_w$ of coprime orders. Then $G$ is soluble, and $w(G)$ is nilpotent.
\end{pro}

\begin{proof} By Proposition \ref{solu}, it is enough to show that $G$ is soluble. Suppose that $G$ is not soluble. Of course, we may assume that $G$ is a counterexample of minimal order. Then every proper subgroup $K$ of $G$ is soluble: indeed, $K/w(K)$ is nilpotent \cite[6.1 Satz]{hup} and so is $w(K)$ by Proposition \ref{solu}. It follows that $G=G'$. 

Let $R$ be the soluble radical of $G$, that is, the subgroup of $G$ generated by all normal soluble subgroups of $G$. Then $G/R$ is a nonabelian simple group and, by Proposition~\ref{solu}, $w(R)$ is nilpotent. We claim that $R=Z(G)$. Choose $q\in\pi(F(G))$ and let $Q$ be the Sylow $q$-subgroup of $F(G)$. According to Lemma \ref{generators}, the group $G$ is generated by $w$-values of $p$-power order for primes $p\neq q$. Also, by Lemma \ref{bbb},  $[Q, x]=1$ for every $w$-value $x$ of $q'$-order. Thus $Q\leq Z(G)$. This happens for each choice of $q\in\pi(F(G))$, so that $F(G)= Z(G)$.  Since $w(R)\leq F(G)$, we have $[x_{i_1}, \ldots, x_{i_k}, y]=1$ for every $x_{i_1}, \ldots, x_{i_k}, y \in R$. Hence $R$ is nilpotent \cite[6.1 Satz]{hup} and therefore $R\leq F(G)$. 
In particular, $R=Z(G)$.   

Next, we prove that $G$ contains a $w$-value $x$ such that $x$ is a 2-element of order $2$ modulo $Z(G)$. First, notice that $G/Z(G)$ is a finite minimal simple group. Then, by Proposition~\ref{msg}, $G/Z(G)$ has a subgroup \[H/Z(G) = A/Z(G) \rtimes T/Z(G)\] where $A/Z(G)$ is an elementary abelian $2$-group and $T/Z(G)$ is a group of odd order such that \[ A/Z(G) = [A/Z(G), \, _{k-1} T/Z(G)].\] Let $P$ be the Sylow $2$-subgroup of $A$. Thus $x=[a, \, _{k-1} t]$,  for some $a \in P$ and $t \in T$, is a $w$-value as desired.

Now, take $x\in G_w$ with the above properties. By Lemma \ref{rembearsuz}, there exists an element $g\in G$ such that the order of $[x,g]$ is an odd prime. Since \[1=[x^2,g]=[x,g]^x[x,g],\] $x$ inverts $[x,g]$ and so, by Lemma \ref{bbb}, $[\langle [x,g]\rangle , x]=1$. This gives $[x,g]=1$, which is a contradiction.
\end{proof}

\section{Proof of Corollary \ref{res}}\label{procor}

Following \cite{dms}, we say that a word $w$ implies virtual nilpotency if every finitely generated metabelian group, where $w$ is a law, has a nilpotent subgroup of finite index. Since finitely generated $n$-Engel groups are nilpotent (see \cite[Part 2, Theorem 7.3.5]{rob2}), the Engel words imply virtual nilpotency. More generally, this is true for simple commutator words. 

\begin{lem}\label{simple}
Let $w=[x_{i_1}, \ldots, x_{i_n}]$ be a simple commutator word and let $G$ be a metabelian group such that $w(G)=1$. Then $G$ is $n$-Engel.
\end{lem}

\begin{proof}
Since $G$ is metabelian, we have $[c,x_{i_j},x_{i_k}]=[c,x_{i_k},x_{i_j}]$ for every $c\in G'$. Then, without loss of generality, we may assume that 
\[w=[x_{i_1}, \ldots, x_{i_m},\,_{n-m} x_{i_1}]\] 
where $1<m<n$ and $i_1\neq i_j$ for every $j\in \{2,\dots,m\}$. Thus, for any $x,y \in G$, taking $x_{i_1}=y[x,y]$ and $x_{i_2}=\dots=x_{i_m}=y$, we have $[x, \,_{n} y]=1$
and therefore $G$ is $n$-Engel.
\end{proof}

\begin{cor}\label{Engel}
Every simple commutator word implies virtual nilpotency.
\end{cor}

A word $w$ is said to be boundedly concise in a class of groups $\mathcal{C}$ if for every integer $m$ there exists a number $\nu = \nu(\mathcal{C},w,m)$ such that whenever $|G_w| \leq m$ for a group $G \in \mathcal{C}$ it always follows that $|w(G)| \leq \nu$. According to Theorem 1.2 of \cite{dms}, words implying virtual nilpotency are boundedly concise in residually finite groups. This, together with Corollary \ref{Engel}, yields the following corollary.

\begin{cor}\label{simplebound}
Every simple commutator word is boundedly concise in the class of residually finite groups.
\end{cor}

We are finally ready to prove Corollary \ref{residual}: {\em Let $w = [x_{i_1},\ldots, x_{i_k}]$ be a simple commutator word with $i_{1} \neq i_j$ for every $j \in \{2, \ldots, k\}$, and let $G$ be a residually finite group in which $G_w$ is finite. Then $w(G)$ is finite, and it is nilpotent if and only if $|ab| = |a||b|$ for any $a, b \in G_w$ of coprime orders.}

\begin{proof}[{\bf Proof of Corollary \ref{res}}]
Of course $w(G)$ is finite, by Corollary \ref{Engel}. Let us show that $w(G)$ is nilpotent, whenever $|ab| = |a||b|$ for any $a, b \in G_w$ of coprime orders. The converse is clear. 

Since $G$ is residually finite, there exists a normal subgroup $N$ of $G$ such that $N \cap w(G)=1$ and $G/N$ is finite. Notice that for any $w$-value $xN\in G/N$, we have $x\in G_w$ and $|xN|=|x|$. It follows that $G/N$ satisfies the hypotheses of Theorem~\ref{main} and therefore $w(G/N)\simeq w(G)$ is nilpotent.    
\end{proof}

\section{Examples}\label{Examples}

In this section we collect some examples showing that Theorem \ref{main} does not hold for an arbitrary commutator word.

\begin{exa}
Let $w=[x,y]^3$ and let $G=(S \times S) \rtimes C$ where $S$ is the symmetric group of degree $3$, $C= \langle g \rangle$ is the cyclic group of order $2$, and the action is given by $(a,b)^g=(b,a)$ for every $(a,b) \in S \times S$. Then every nontrivial $w$-value has order $2$, and $w(G)=S \times S$.
\end{exa}

\begin{exa}
Let $w=[x,y^{10},y^{10},y^{10}]$ and let $G$ be the alternating group of degree $5$. Then $G_w$ consists of the identity and all products of two transpositions. In particular, $w(G)=G$. 
\end{exa}

\begin{proof}
Of course $G_w$ is the set of all commutators $[g,h,h,h]$, where $g,h \in G$ and $h$ is a 3-cycle. Since $[g,h,h]=(h^{-1})^{[g,h]}h$, we have
\[
[g,h,h,h]=[(h^{-1})^{[g,h]},h]^h= [(h^{-1})^{{(h^{-1})^g}h},h]^h=[(h^{-1})^k,h]^{h^{-1}},
\]
where $k=(h^{-1})^g$ is a $3$-cycle. For any $3$-cycles $h,k \in G$, we claim  that $[(h^{-1})^k,h]$ is either trivial or a product of two transpositions, from which it follows that so is $[(h^{-1})^k,h]^{h^{-1}}$.
	
Let $h=(a\ b\ c)$ and $k=(d\ e\ f)$. Clearly, we may assume $a=d$. Furthermore, it is enough to consider the cases:
\begin{enumerate}
\item[(1)] $b=e$ and $c\neq f$;
\item[(2)] $e,f \not\in\{b, c\}$.
\end{enumerate}
	
\noindent In fact, the other (nontrivial) cases can be deduced applying the following identities:
$$[(h^{-1})^k,h]^{-(h^{-1})^k h^{-1} k^{-1}}=[((h^{-1})^k)^{-1},h^{-1}]^{-k^{-1}}=[(h^{-1})^{k^{-1}},h];$$
$$[(h^{-1})^k,h]^{h^k k^{-1}}=[(h^k)^{-1},h]^{h^k
k^{-1}}=[h^k,h]^{-k^{-1}}=[h^{k^{-1}},h].$$
	
Now, in the first case, $h$ and $k$ belong to the alternating group of degree $4$ and therefore $[(h^{-1})^k,h]$ is the product of two transpositions. In the second case, we have
$$[(h^{-1})^k,h]=[(a\ c\ b)^{(a\ e\ f)},(a\ b\ c)]=[(b\ e\ c), (a\ b\ c)]=(a\ c)(b\ e).$$
This proves our claim. Also, it implies that $G_w$ contains all products of two transpositions.
\end{proof}

\begin{exa}[see Theorem 1.2 of \cite{kani}]\label{v10}
For every $n\geq 7$, there exists a commutator word $v$ such that, for $w=v^{10}$, the set of $w$-values of the alternating group of degree $n$ consists of the identity and all $3$-cycles.
\end{exa}


\begin{thebibliography}{10}
		
\bibitem{bms} R. Bastos, C. Monetta and P. Shumyatsky, {\it A criterion for metanilpotency of a finite group}, J. Group Theory {\bf 21} (2018), no. 4, 713--718.

\bibitem{BS} R. Bastos and P. Shumyatsky, {\it A sufficient condition for nilpotency of the commutator subgroup}, Sib. Math. J. {\bf 57} (2016), no. 5, 762--763.  

\bibitem{baubau} B. Baumslag and J. Wiegold, {\it A sufficient condition for nilpotency in a finite group}, arXiv: 1411.2877 [math.GR] (2014).

\bibitem{bloom} D.\,M. Bloom, {\it The subgroups of $\Psl(3,q)$ for odd $q$}, Trans. Amer. Math. Soc. {\bf 127} (1967), 150--178.

\bibitem{dms} E. Detomi, M. Morigi and P. Shumyatsky, {\it Words of Engel type are concise in residually finite groups}, arXiv: 1711.04866v2 [math.GR] (2017).

\bibitem{go} D. Gorenstein, {\it Finite Groups}, 2nd Edition, Chelsea Publishing Co., New York, 1980.

\bibitem{hup} B. Huppert, {\it Endliche Gruppen I}, Springer-Verlag, Berlin-New York, 1967.

\bibitem{Isaacs} I.\,M. Isaacs, {\it Finite group theory},  Graduate Studies in Mathematics 92, American Mathematical Society, Providence, RI, 2008.

\bibitem{kani} M. Kassabov and N. Nikolov, {\it Words with few values in finite simple groups}, Q. J. Math. {\bf 64} (2013), no. 4, 1161--1166.

\bibitem{rob2} D.\,J.\,S. Robinson, {\it Finiteness conditions and generalized soluble groups}, Parts 1 and 2, Springer-Verlag, Berlin, 1972.

\bibitem{suzuki} M. Suzuki, {\it Group Theory I}, Springer-Verlag, Berlin-New York, 1982.

\bibitem{suzuki2} M. Suzuki, {\it On a class of doubly transitive groups}, Ann. of Math. (2) {\bf 75}  (1962), 105--145.

\bibitem{thompson} J.\,G. Thompson, {\it Nonsolvable finite groups all of whose local subgroups are solvable}, Bull. Amer. Math. Soc. \textbf{74} (1968), 383--437. 

\bibitem{turull} A. Turull, { \it Fitting height of groups and of fixed points}, J. Algebra, \textbf{86} (1984), no. 2, 555--566. 

\end{thebibliography}
\end{document}